\documentclass{amsart}
\usepackage{a4wide,amsmath,amsfonts,amssymb,amsthm,mathtools}
\usepackage{amsfonts}
\usepackage{mathtools,amssymb}
\usepackage{amsmath}
\usepackage{amsthm}
\usepackage{graphicx} 
\usepackage{wrapfig}
\usepackage{xcolor}
\usepackage[utf8]{inputenc}
\usepackage[margin=1.5cm]{geometry}
\usepackage{enumerate}
\usepackage{hyperref}

\usepackage{a4wide,amsmath,amsfonts,amssymb,amsthm,mathtools}
\usepackage[all]{xy}
\usepackage{cancel}
\usepackage[normalem]{ulem}

\hypersetup{
    colorlinks=true,
    linkcolor=blue,
    filecolor=magenta,      
    urlcolor=cyan,
    pdftitle={Overleaf Example},
    pdfpagemode=FullScreen,
    }

\newcommand{\R}{\mathbb{R}}
\newcommand{\on}{\operatorname}
\newcommand{\Tau}{\mathcal{T}}
\newcommand{\G}{\mathcal{G}}
\newcommand{\vr}{\varrho}

\newcommand{\N}{\mathbb{N}}
\newcommand{\D}{\mathcal{D}}
\newcommand{\K}{\mathcal{K}}

\renewcommand{\S}{\mathcal{S}}
\newcommand{\F}{\mathcal{F}}
\newcommand{\KF}{\mathcal{K}_\mathcal{F}}
\newcommand{\acut}[2][\alpha]{\left[{#2}\right]^{#1}}

\newtheorem{theorem}{Theorem}[section]
\newtheorem{proposition}{Proposition}[section]
\newtheorem{lemma}{Lemma}[section]
\newtheorem{definition}{Definition}[section]
\newtheorem{remark}{Remark}[section]
\theoremstyle{remark}

\theoremstyle{definition}
\newtheorem{exmp}{Example}[section]

\title{A topological approach to fuzzy iterated function systems}

\author{Taras Banakh, Krzysztof Caban, Filip Strobin}
\address{T. Banakh: Ivan Franko National University of Lviv (Ukraine), and Jan Kochanowski University in Kielce (Poland); ORCID 0000-0001-6710-4611}
\email{t.o.banakh@gmail.com}
\address{K. Caban: {Institute of Mathematics, University of Gdańsk, Wita Stwosza 57
80-308 Gdańsk} (Poland)}
\email{kcaban718@gmail.com}
\address{F. Strobin: {Institute of Mathematics, University of Gdańsk, Wita Stwosza 57
80-308 Gdańsk} (Poland);
ORCID 0000-0002-8671-9053}
\email{{filip.strobin@ug.edu.pl}}


\begin{document}

\maketitle

\begin{abstract}
    In the paper we unify two extensions of the classical Hutchinson--Barnsley theory - the topological and the fuzzy-set approaches. We show that a fuzzy iterated function system (fuzzy IFS) on a Tychonoff space $X$ which is contracting w.r.t. some admissible multimetric, generates a natural fuzzy attractor in the hyperspace $\KF(X)$ of all compact fuzzy sets. As a consequence, we prove that a fuzzy IFS on a Hausdorff topological space which is topologically contracting admits a fuzzy attractor in a bit weaker sense. Our discussion involves investigations on topologies on the hyperspace $\KF(X)$ which are suitable for establishing convergence of sequences of iterations of a fuzzy Hutchinson operator.    
\end{abstract}

\section{Introduction}
One of the most popular ways of generating fractal sets and modelling related real-life objects bases on the famous Hutchinson theorem \cite{B}, which allows to obtain such sets as contractive fixed points of Hutchinson operators generated by iterated function systems (IFSs) comprised of contractive maps. The original, very restrictive setting of Banach contractions on complete metric spaces has been extended in many ways, which resulted in widening the range of objects that can be generated and investigated from the perspective of the fractal theory.

One of such extensions has been introduced by Mihail in \cite{M1} (but motivated by earlier works, for example \cite{Kameyama}), in which the metric contractivity is replaced  by so-called \emph{topological contractivity} considered in all Hausdorff topological spaces. Topological contractivity of an IFS means that any compact set can be covered by subinvariant compact set (\emph{compact dominacy}) and that all its ,,fibers'' (intersections of images of compositions of maps driven by a sequence in the code space) are singletons. Mihail proved that topologically contracting IFSs generate (necessarily unique) attractors, that is, invariant sets which are limits of iterations of the Hutchinson operators w.r.t. the Vietoris topology. This setting was further studied by Banakh et.al. in \cite{BKNNS} (see also \cite{MM}), in which it was shown that among a wide range of Tychonoff spaces, topological contractivity can be explained by contractivity w.r.t. some admissible multimetric (that is, a family of pseudometrics that separates points and which generate the original topology), and that compact dominating IFSs which are contractive w.r.t. some multimetric, generate attractors. This allowed to explain the Mihail result by contractivity w.r.t. multimetrics and revealed a deep connection between topological contractivity and (pseudo) metric contractivity.

The second extension of the classical IFS theory involves the framework of fuzzy sets, which itself have many applications. In this approach, first presented in \cite{Cabrelli}, classical IFSs endowed with certain families of selfmaps of the unit interval $[0,1]$, generate counterparts of Hutchinson operators on the hyperspace of ,,compact'' fuzzy sets. Contractive fixed points of these operators (being certain fuzzy sets) are called \emph{fuzzy attractors}. It turns out that classical results on the Hutchinson--Barnsley theory have natural extensions in such a setting. For example, contractive fuzzy IFSs on complete metric spaces generate fuzzy attractors (see i.e. \cite{Cabrelli}, \cite{OS}), that can be approximated with the use of counterparts of classical algorithms (see \cite{KSW}), and which can be described with a help of the projection map from the code space (see \cite{Miculescu}).

The aim of this paper is to unify these two approaches and study fuzzy IFSs on topological spaces, which are topologically contracting or contracting w.r.t. admissible multimetrics.

In the main results, we show that such fuzzy IFSs generate fuzzy attractors. The important part of our study is focused on detecting appropriate topologies on the hyperspace of ,,compact'' fuzzy sets. These topologies should be natural, strong enough, but at the same time they should allow the convergence of iterations of the fuzzy Hutchinson operators to fuzzy attractors. In the case of Tychonoff spaces, they should be independent of the choice of admissible multimetric.\\$\;$

The paper is organized as follows.

In Section \ref{sec:2} we recall notions and facts from the classical IFS theory (Subsections \ref{sec:ifs}, \ref{sec:HB}), from its extensions to topological and multimetric spaces (Subsections \ref{sec:TC}, \ref{sec:Mult}, \ref{sec:multimetricIFS}), and from the fuzzy-set approach to the IFS theory (Subsection \ref{sec:fuzzy}). In Subsections \ref{sec:Mult}, \ref{sec:multimetricIFS} and \ref{sec:fuzzy} we also prove several auxiliary results that we use later on.

In Section \ref{sec:3} we study fuzzy IFSs on multimetric spaces. In Subsection \ref{sec:conon} we define the hyperspace of ,,compact'' fuzzy sets $\KF(X)$ on multimetric space $(X,\D)$ and introduce the canonical topology $\Tau_\D^\F$ on it. This topology is induced by a family of fuzzy Hausdorff-type pseudometrics, which naturally extends the fuzzy Hausdorff-type metric considered for the metric case (see \cite{Cabrelli}, \cite{OS}), and is independent on the choice of an admissible multimetric. In Subsection \ref{sec:2.3} we define a fuzzy attractor of a fuzzy IFS on a multimetric space and prove (Theorem \ref{thm:main}) that compactly dominating fuzzy IFSs that are contracting w.r.t. some admissible multimetric generate fuzzy attractors. This result is one of the main results of the paper. In Subsection \ref{sec:Tych} we slightly reformulate the definition of a fuzzy attractor and adjusts it to IFSs on Tychonoff spaces. We also reformulate Theorem \ref{thm:main} in this spirit (Theorem \ref{thm:main3}) and give its corollary (Theorem \ref{thm:main4}) which states that topologically contracting fuzzy IFSs on Tychonoff $k$-spaces generate fuzzy attractors. This result is also one of the main results of the paper.

In Section \ref{sec:other} we consider the fuzzy IFS theory on Hausdorff topological spaces. In Subsection \ref{sec:hypo} we introduce and study the topology $\Tau_h$ induced by identifying ,,compact'' fuzzy sets with their hypographs as elements of the hyperspace of nonempty and compact subsets of $X\times[0,1]$. We prove that for Tychonoff spaces, this topology is weaker than the 
canonical topology considered previously. However, in Subsection \ref{sec:TFIFS} we prove (Theorem \ref{thm:main10}) that fuzzy IFSs on Hausdorff topological spaces that are topologically contracting,  admit the invariant ,,compact'' fuzzy set, which are limits of iterations of fuzzy operators w.r.t. the topology $\Tau_h$ (we call such set as weak fuzzy attractors). This result is also one of the main results of the paper. Then in Subsection \ref{sec:descr} we give descriptions of a (weak) fuzzy attractor, which involve its description for a metric case from \cite{Miculescu}, together with the fact that fuzzy attractors are natural extension of fuzzy attractors of restricted fuzzy IFSs.

Finally, in Section \ref{sec:5} we consider another topologies on $\KF(X)$. In Subsection \ref{sec:5.1} we define the topology $\Tau_{h_0}$ by identifying ,,compact'' fuzzy sets with an alternative version of hypographs. We show that $\Tau_{h_0}$ is weaker than $\Tau_h$, which implies that $\Tau_h$ is better than $\Tau_{h_0}$ from the perspective of convergence of iterations to a fuzzy attractor. On the other hand, we show some benefits when restricting to the topology $\Tau_{h_0}$. Finally, in Subsection \ref{sec:unif}, we show that the topologies of uniform convergence (i.e., generated by the $\ell_\infty$-metric) or of pointwise convergence, are not appropriate for our needs.

\section{Preliminaries and basic results}\label{sec:2}
If $(X,\Tau)$ is a topological space, then by $\K(X)$ we denote the family of all nonempty and compact subsets of $X$, endowed with the Vietoris topology $\Tau_V$. It is known (cf. \cite[3.12.27(a)]{En}) that $(\K(X),\Tau_V)$ is Hausdorff (and compact) iff $X$ is Hausdorff (and compact). 
If $(X,d)$ is a metric space, then we also endow the hyperspace $\K(X)$ with the Hausdorff (also called as the Hausdorff-Pompeiu) metric, denoted by $d_H$. It is also known (see \cite[4.5.23]{En}) that it generates the Vietoris topology and that $(X,d)$ is complete iff $(\K(X),d_H)$ is complete.
We also have the following lemma which seems to be a folklore, but we state a short proof for clarity.

\begin{lemma}\label{lem:subViet}
If $Y$ is a subspace of a topological space $X$, then the Vietoris topology on $\K(Y)$ is exactly the induced topology on $\K(Y)$ from $\K(X)$.
\end{lemma}

\begin{proof} The result follows from the next two equalities holding for all open sets  $U\subseteq X$:     $$
     \begin{aligned}
    &\{K\in\K(X):K\subset U\}\cap\K(Y)=\{K\in\K(Y):K\subset U\cap Y\}\quad\mbox{and}\\
    &\{K\in\K(X):K\cap U\ne\emptyset\}\cap\K(Y)=\{K\in\K(Y):K\cap(U\cap Y)\neq \emptyset\}.
    \end{aligned}
    $$
    
\end{proof}

For a subset $C$ of a topological (or metric) space $X$, by $\overline{C}$ we will denote the closure of $C$. For a function $f:X\to X$ and $K\subset X$, by $f{\restriction}_{ K}$ we will denote the restriction of $f$ to the set $K$, and $f^{(n)}$ will stand for the $n$-fold composition of $f$, where $n\in\N:=\{1,2,3,..\}$. For a set $C\subset X$, by $\chi_C$ we denote the characteristic function of $C$, that is
$$
\chi_C(x)=\left\{\begin{array}{ccc}1&\mbox{if}&x\in C\\0&\mbox{if}&x\notin C\end{array}\right..
$$
If $C=\{c\}$ is a singleton, then $\chi_{\{c\}}$ is denoted by $\chi_c$ and is called the {\em Dirac function} of the point $c$.

\subsection{Iterated function systems and their attractors}\label{sec:ifs}
In a standard setting, iterated function systems and their attractors are considered for metric spaces. We will provide a wider approach - for Hausdorff topological spaces. 
\begin{definition}\label{def:2.1}\emph{
    Any system $\S=(f_i)_{i=1}^k$ of finitely many continuous selfmaps of a Hausdorff topological space $X$ is called an } iterated function system \emph{(IFS for short).}\\
    \emph{Any IFS $\S=(f_i)_{i=1}^k$ generates the so-called } Hutchinson operator\emph{, that is the map $\S:\K(X)\to\K(X)$ defined by}
    \[
        \S(K):=\bigcup_{i=1}^kf_i(K).
    \]
    \emph{A (necessarily unique) set $A_\S\in\K(X)$ is called} an attractor of an IFS $\S$\emph{ if the following conditions holds:
    \begin{itemize}
        \item[(1)] $A_\S=\S(A_\S)=\bigcup_{i=1}^{k}f_i(A_\S)$;
        \item[(2)] for every $K\in\K(X)$, the sequence of iterations $(\S^{(n)}(K))$ converges to $A_\S$ w.r.t. the Vietoris topology.
    \end{itemize}}
\end{definition}

\subsection{Metric contractivity and the Hutchinson--Barnsley theorem}\label{sec:HB}
We will introduce several notions of contractivity, among them the ones which generalize the standard Banach contractivity.

We say that a map $\varphi:[0,\infty)\to[0,\infty)$ is \emph{a comparison function}, if it is nondecreasing and for every $t\geq 0$, the sequence of iterations $\varphi^{(n)}(t)\to 0$.
\begin{definition}\label{def:contractions}\emph{
Assume that $(X,d)$ is a metric space and $f:X\to X$. We say that $f$ is} 
\begin{itemize}
    \item[$\bullet$] a Banach contraction, \emph{if its Lipschitz constant $\on{Lip}(f)<1$;}
    \item[$\bullet$] a Matkowski contraction, \emph{if for some comparison function $\varphi$, it holds
    \begin{equation}\label{matcontr}
      \forall_{x,y\in X}\;d(f(x),f(y))\leq \varphi(d(x,y));  
    \end{equation}
    the map $\varphi$ will be called as} a witness \emph{for Matkowski contractivity of $f$;}
    \item[$\bullet$] an Edelstein contraction, \emph{if 
   $$
    \forall_{x,y\in X,x\neq y}\;d(f(x),d(y))<d(x,y).
   $$
   }
\end{itemize}
\end{definition}
\begin{remark}\label{rem:contractions}\emph{
    Clearly, each Banach contraction is a Matkowski contraction, and, as a comparison function $\varphi$ necessarily satisfies $\varphi(t)<t$ for $t>0$, each Matkowski contraction is also an Edelstein contraction. Moreover, none of these implications can be reversed, yet if $X$ is compact, then the Edelstein contractivity is equivalent to the Matkowski contractivity.\\
    Let us also note that by the Matkowski fixed point theorem, each Matkowski contraction on a complete metric space satisfies the thesis of the Banach fixed point theorem. This result can be considered as one of the widest generalizations of the Banach's theorem, as Matkowski contractivity implies many others, like those due to Rakotch or Browder.\\
Finally, note that Edelstein contractivity is not sufficient for the existence of a fixed point - simply consider the map $x\to x+e^{-x}$, for $x\in X:=[0,\infty)$.\\
For a brief exposition of various types of contractivity and their mutual relationships, we refer the reader to the paper \cite{JJ}.}
\end{remark}

\begin{definition}
    \emph{Let $\S=(f_i)_{i=1}^k$ be an IFS on a metric space $(X,d)$. We say that $\S$ is }Matkowski [Banach, Edelstein] contracting\emph{,} if the maps $f_1,...,f_k$ are Matkowski [Banach, Edelstein, respectively] contractions.
\end{definition}

The classical Hutchinson theorem from early 1980s states that Banach contracting IFSs on complete metric spaces admit (necessarily unique) attractors. This result has been extended to weaker types of contractivity, in particular for Matkowski contracting IFSs (see, e.g., \cite{LSS}).

\begin{theorem}\label{MIFS}
    Each Matkowski contracting IFS on a complete metric space admits an attractor.
\end{theorem}
\begin{remark}\label{rem:maxvar}\emph{The proof of Theorem \ref{MIFS} can be based on the facts that if $\varphi_i$, $i=1,...,k$, are comparison functions, then $\varphi:=\max\{\varphi_i:i=1,...,k\}$ is also a comparison function, and if maps of the underlying IFS $\S$ are Matkowski contractions with a common witness $\varphi$, then the Hutchinson operator $\S:\K(X)\to\K(X)$ is a Matkowski contraction with a witness $\varphi$.}
\end{remark}

\subsection{Topological contractivity}\label{sec:TC}
Now let us recall the notion of topological contractivity, introduced by Mihail in \cite{M1}.

For a natural number $k\in\N$, denote by $\Sigma_k$ the family of all sequences of numbers from $\{1,...,k\}$, that is, $\Sigma_k=\{1,...,k\}^\N$.

\begin{definition}\emph{
    Let $X$ be a Hausdorff topological space and $\S=(f_i)_{i=1}^k$ be an IFS on $X$.\\
    We say that $\S$ is }compactly dominating \emph{, if for every $K\in\K(X)$, there exists $D\in\K(X)$ such that $K\subset D$ and $\S(D)\subset D$.\\
    We say that $\S$ is }topologically contracting \emph{, if it is compactly dominating and for every $D\in\K(X)$ with $\S(D)\subset D$ and every sequence $(\sigma_n)\in\Sigma_k$, the set
    \begin{equation}\label{eq:wee}
       \bigcap_{n\in\N}f_{\sigma_1}\circ...\circ f_{\sigma_n}(D) 
    \end{equation}
    is a singleton.
    }
\end{definition}

\begin{remark}\label{rem:proj}\emph{
It is easy to see that each Matkowski contracting IFS on a complete metric space is topologically contracting. Compact dominancy follows from the existence of an attractor for such an IFS, as for each $K\in\K(X)$, the set 
\[
    D:=K\cup \bigcup_{n\in\N}\S^{(n)}(K)\cup A_\S=\overline{K\cup \bigcup_{n\in\N}\S^{(n)}(K)}
\]
is compact and $\S(D)\subset D$.\\
The intersection (\ref{eq:wee}) is singleton as the intersection of decreasing sequence of compact sets with diameters tending to zero.}
\end{remark}
\begin{remark}\label{rem:prof2}
    \emph{Note that if $\S=(f_i)_{i=1}^k$ is topologically contracting, then
    the intersection (\ref{eq:wee}) does not depend on the choice of $D$. Denoting its unique element by $\pi(\sigma)$, where $\sigma=(\sigma_n)$, we define the map $\pi:\Sigma_k\to X$ which is called the } projection map.\emph{ It is an essential tool in studying the structure of IFS attractors. In this context, the space $\Sigma_k$ is called } a code space\emph{ or }an address space\emph{. The details can be found, for example in \cite{LSS} or in} other classical books on the IFS theory.
\end{remark}

Mihail \cite{M1} proved that topologically contracting IFSs generate attractors, which, as observed in \cite{BKNNS}, are metrizable.
\begin{theorem}\label{TIFS}(\cite[Theorem 3.1]{M1}, \cite[p. 1031]{BKNNS})\\
Each topologically contracting IFS $\S$ on a Hausdorff  space $X$ admits an attractor, which is a metrizable subset of $X$.
\end{theorem}
Due to the previous remark, this result is an extension of Theorem \ref{MIFS}. However, it turns out that if $X$ is metrizable, then Theorems \ref{MIFS} and \ref{TIFS} are in some sense equivalent, as the existence of attractors of topologically contracting IFSs can be fully explained by Theorem \ref{MIFS} according to remetrization techniques presented in \cite{BKNNS} and \cite{MM}. In Section \ref{sec:multimetricIFS} we will formulate a more general result on these issues.

\subsection{Multimetric spaces}\label{sec:Mult}
In this section we recall the pseudometric approach for the IFS theory, introduced in \cite{BKNNS}.\\
Recall that a \emph{pseudometric} $d$ on a set $X$ is a map that satisfies the conditions of a metric but without the full first one; that is, for $x,y,z\in X$, we have $d(x,x)=0$, $d(x,y)=d(y,x)$ and $d(x,y)\leq d(x,z)+d(z,y)$. Note that each pseudometric $d$ is nonnegative, which follows from the triangle inequality and symmetry: $0=d(x,x)\le d(x,y)+d(y,x)=2d(x,y)$.

\begin{definition}
    \emph{By a} multipseudometric on a set $X$\emph{ we mean  a family $\D$ of pseudometrics on $X$. It is called} a multimetric,\emph{ if additionally for every distinct $x,y\in X$, we find $d\in\D$ so that $d(x,y)>0$.
    Then the pair $(X,\D)$ is called} a multimetric space.
    \emph{A multimetric $\D$ is said to be} directed\emph{, if for every nonempty and finite $\G\subset\D$, we find $d\in\D$ such that $d\geq\max \G$.}
    
    \emph{Each multipseudometric $\D$ on $X$ generates the canonical topology $\Tau_\D$ consisting of all sets $U\subset X$ so that for every $x\in U$, there exists $\varepsilon>0$ and a finite family $\G\subset \D$ such that}
    \[
        B_\G(x,\varepsilon):=\bigcap_{d\in\G}B_d(x,\varepsilon)=\{y\in X:\forall_{d\in \G}\;d(x,y)<\varepsilon\}\subset U,
    \]
    \emph{where $B_d(x,\varepsilon):=\{y\in X:d(x,y)<\varepsilon\}$. By the triangle inequality, each set $B_d(x,\varepsilon)$ is open. Moreover, it is easy to see that $\Tau_\D$ is the coarsest topology in which all pseudometrics from $\D$ are continuous.}
\end{definition}
\begin{remark}\label{rem:multimetric}
    \emph{
    Note that the topology $\Tau_\D$ generated by a multimetric is Tychonoff (in particular, Hausdorff), as each multimetric generates a natural uniformity, and conversely, each Tychonoff topology is generated by some multimetric $\D$ (see \cite[Section 8.1]{En}). Thus, discussing topological spaces $\Tau_\D$ restricts the framework to Tychonoff spaces. Moreover, each compact Hausdorff space is Tychonoff, so each compact Hausdorff topology is generated by some multimetric. Finally, it is clear that if $\D$ is an admissible multimetric on a topological space $X$ and $\emptyset\neq K\subset X$, then the family of restrictions $\D_K:=\{d{\restriction}_{ K\times K}:d\in\D\}$ is an admissible multimetric on topological subspace $K$.}
    
\end{remark}
It turns out (see \cite[p. 1038]{BKNNS}) that each multimetric $\D$ on a set $X$ generates the multipseudometric $\D_H=\{d_H:d\in\D\}$ on the hyperspace $\K(X)$ by
\begin{equation}\label{eq:fdhh}
    d_H(K,D):=\max\{\max_{x\in K}\min_{y\in D}d(x,y),\max_{y\in D}\min_{x\in K}d(x,y)\}.
\end{equation}
\begin{theorem}\label{directed multimetrics}(\cite[p. 1038 and Proposition 3.2]{BKNNS})
    \begin{itemize}
        \item[(1)] If $\D$ is a multimetric, then
        $
            \tilde{\D}:=\{\max\G:\G\subset\D\mbox{ is nonempty and finite}\}
        $
        is a directed multimetric and $\Tau_{\D}=\Tau_{\tilde{\D}}$.
        \item[(2)] If $\D$ is a directed multimetric, then $\D_H$ is a multimetric and generates the Vietoris topology on $\K(X)$.
    \end{itemize}
\end{theorem}
\begin{remark}\label{rem:directa}
\emph{Let us note that the assumption that $\D$ is directed is important. According to \cite[Example 3.1]{BKNNS}), there is a nondirected multimetric $\D$ that generates the Euclidean topology on $X=\R^2$ such that $\Tau_{\D_H}$ is not the Vietoris topology.}
\end{remark}
We end this section with two lemmas that will be useful later on.
\begin{lemma}\label{lem:mult1}
Assume that $\D,\G$ are directed multimetrics on the same set $X$. The following conditions are equivalent:
    \begin{itemize}
            \item[(i)] $\Tau_\D\subset\Tau_\G$.
            \item[(ii)] $\forall_{x\in X}\;\forall_{d\in \D}\;\forall_{\varepsilon>0}\;\exists_{g\in\G}\;\exists_{\delta>0}\;B_g(x,\delta)\subset B_d(x,\varepsilon)$.
    \end{itemize}
\end{lemma}
\begin{proof}
    The result follows from a simple fact that if $\D$ is a directed multimetric, then for each $x\in X$, the family $\{B_d(x,\varepsilon):d\in\D,\varepsilon>0\}$ is a base at the point $x$.
\end{proof}

\begin{lemma}\label{lem: unversal pseudometric for compact set}
        Let $\D,\G$ be directed multimetrics on the same set $X$. If $\Tau_\D = \Tau_\G$ and $K\subset X$ is compact, then 
        \[
            \forall_{d\in\D}\;\forall_{\varepsilon>0}\;\exists_{g\in\G}\;\exists_{\delta>0} \;\forall_{x\in K}\;\  B_g(x, \delta)\subset B_d(x,\varepsilon).
        \]
    \end{lemma}
    \begin{proof}
        Fix $d\in \D$ and $\varepsilon>0$. By the previous lemma, for every $x\in K$, there exist $g_x\in\G$ and $\delta_x>0$ such that
        \begin{equation}\label{aa22}
            B_{g_x}(x,\delta_x)\subset B_d\big(x,\frac{\varepsilon}{2}\big).
        \end{equation}
        Since the family $\{B_{g_x}(x,\frac{\delta_x}{2}):\ x\in K\}$ is an open cover of $K$ and $K$ is compact, we can choose a finite subcover $B_{g_{x_i}}(x_i, \frac{\delta_{x_i}}{2}):i=1,...,k$ of $K$. Since $\G$ is directed, we can pick $g\in \G$  with $g\geq\max\{g_{x_i}:i=1,...,k\}$. Fix $\delta := \frac{1}{2}\min \{{\delta_{x_i}}:i=1,...,k\}$. We will show that for any $z\in K$,
        \[
            B_g(z,\delta)\subset B_d(z,\varepsilon).
        \]
        Take any $z\in K$. Then there exists $i=1,...,k$ such that $z\in B_{g_{x_i}}\big(x_i, \frac{\delta_{x_i}}{2}\big)$. In particular, by (\ref{aa22}) we have $d(z, x_i) < \frac{\varepsilon}{2}$. Let $y\in B_g(z,\delta)$. Then
        \begin{align*}
            g_{x_i}(y,x_i)\leq g_{x_i}(y, z) + g_{x_i}(z, x_i) < g(y,z) + \frac{\delta_{x_i}}{2}<
             \delta + \frac{\delta_{x_i}}{2}\leq \frac{\delta_{x_i}}{2} + \frac{\delta_{x_i}}{2} = \delta_{x_i}.
        \end{align*}
        Hence, $y\in B_{g_{x_i}}\big(x_i, \delta_{x_i}\big)$. Again by (\ref{aa22}), we have
        \[
            d(z,y) \leqslant d(z,x_i) + d(x_i, y) < \frac{\varepsilon}{2} + \frac{\varepsilon}{2} = \varepsilon,
        \]
        and consequently $y\in B_d(z, \varepsilon)$.
    \end{proof}

\subsection{Multimetric contractivity vs topological contractivity}\label{sec:multimetricIFS}

If $d$ is a pseudometric, then we can consider the same types of contractivity as in Definition \ref{def:contractions} with respect to $d$, that is, Banach, Matkowski and Edelstein contractivity (see \cite[Definition 4.1]{BKNNS}). The difference is that the Edelstein contractivity should be defined in the following way:
\[
    \forall_{x,y\in X}\;d(f(x),f(y))\leq d(x,y)\;\;\mbox{and}\;\;\forall_{x,y\in X,d(x,y)>0}\;d(f(x),f(y))<d(x,y)
\]
(above we clarify a bit the formulation given in \cite{BKNNS}).
The relationships between these notions for pseudometrics are the same as for metrics indicated in Remark \ref{rem:contractions}, that is, Banach contractivity implies Matkowski contractivity and Matkowski contractivity implies Edelstein contractivity, which is equivalent to Matkowski contractivity if the space $X$ is compact and pseudometric is continuous, see \cite[Proposition 4.2]{BKNNS}.

If $(X,\D)$ is a multimetric space, then we can consider iterated function systems on $X$ as system $\S=(f_i)_{i=1}^k$ of selfmaps of $X$ which are continuous w.r.t. the topology $\Tau_\D$. It is easy to see that if $f:X\to X$ is nonexpansive w.r.t. to all pseudometrics $d\in\D$ (that is, $d(f(x),f(y))\leq d(x,y)$ for each $x,y\in X$), then $f$ is continuous. Hence systems of such maps are IFSs in the sense of Definition \ref{def:2.1}.

\begin{definition}[{\cite[Definition 5.1 and a comment after it]{BKNNS}}]
\emph{Let $(X,\D)$ be a multimetric space. 
We say that an IFS $\S=(f_i)_{i=1}^k$ is } Matkowski \emph{[}Banach, Edelstein\emph{]} contracting, \emph{if for every $d\in \D$, each $f_i$, $i=1,...,k$, is a Matkowski [Banach, Edelstein, respectively] contraction w.r.t. $d$.} 
\end{definition}

\begin{lemma}\label{rem:2.6}
If $\S$ is Matkowski [Banach, Edelstein] contracting w.r.t. a multimetric $\D$, then it is also Matkowski [Banach, Edelstein, respectively] contracting w.r.t. the directed multimetric $$\tilde{\D}:=\{\max\G:\G\subset \D,\;\G\;\mbox{nonempty and finite}\}.$$
\end{lemma}
\begin{proof}
 We will just show the thesis for Matkowski and Edelstein contractivity. Assume first that $\S=(f_i)_{i=1}^k$ is Matkowski contracting, take any $i=1,...,k$ and a finite and nonempty $\G\subset\D$, and let $\varphi_\G:=\max\{\varphi_d:d\in\G\}$, where $\varphi_d$ is a witness for Matkowski contractivity of $f_i$ for a pseudometric $d$. Then for every $x,y\in X$,
$$
\max\G(f_i(x),f_i(y))=\max\{d(f_i(x),f_i(y)):d\in\G\}\leq\max\{\varphi_d(d(x,y)):d\in\G\}\leq $$ $$\leq \max\{\varphi_d(\max \G(x,y)):d\in\G\}= \varphi_\G(\max\G(x,y)).
$$
Assume now that $\S$ is Edelstein contracting, take any $i=1,...,k$ and a finite and nonempty $\G\subset\D$. For $x,y\in X$, we have
$$
\max\G(f_i(x),f_i(y))=\max\{d(f_i(x),f_i(y)):d\in\G\}\leq\max\{d(x,y):d\in\G\}=\max\G(x,y).
$$
Now if $\max\G(x,y)>0$, then choosing $\G_0:=\{d\in\G:d(x,y)>0\}$, we have
$$
\max\G(f_i(x),f_i(y))=\max\G_0(f_i(x),f_i(y))<\max\G_0(x,y)=\max\G(x,y).
$$
\end{proof}
The following result gathers some observations from \cite{BKNNS} which show certain relationships between topological, Matkowski and Edelstein contractivity, and which base on remetrization techniques from mentioned papers \cite{BKNNS} and \cite{MM}. 
\begin{theorem}\label{aaa}
Let $X$ be a Hausdorff topological space and $\S=(f_i)_{i=1}^k$ be an IFS on $X$.
    \begin{itemize}
        \item[(1)] If $\D$ is an admissible multimetric on $X$ and $\S$ is compactly dominating and Edelstein contracting w.r.t $\D$, then $\S$ is topologically contracting.
        \item[(2)] If $X$ is compact and $\D$ is an admissible multimetric on $X$, then $\S$ is Edelstein contracting iff $\S$ is Matkowski contracting w.r.t $\D$.
        \item[(3)] If $\S$ is topologically contracting, then for every $K\in\K(X)$ with $\S(K)\subset K$, there is an admissible multimetric $\D$ on $K$ so that the restricted IFS $\hat{\S}=(f_i{\restriction}_{ K})_{i=1}^k$ is Matkowski contracting w.r.t. $\D$.
        \item[(4)] If $\S$ is topologically contracting and $A_\S$ is its attractor, then there is an admissible metric $d$ on $A_\S$ so that the restricted IFS $\hat{\S}=(f_i{\restriction}_{A_\S})_{i=1}^k$ is Matkowski contracting w.r.t. $d$.
    \end{itemize}
\end{theorem}
\begin{proof}
Item (1) follows from \cite[Theorem 5.8]{BKNNS}, equivalence $(a)\Leftrightarrow(e)$. Item (2) follows from \cite[Theorem 6.8]{BKNNS}, equivalence $(g)\Leftrightarrow(h)$ and the fact that Rakotch contractiveness implies Matkowski contractiveness, and Matkowski contractiveness implies Edelstein contractiveness (\cite[Proposition 4.2]{BKNNS}). Item (3) follows from \cite[Theorem 6.8]{BKNNS}, equivalence $(a)\Leftrightarrow(g)$ and the fact that the restricted IFS $\hat{\S}$ is also topologically contracting. Finally, item (4) follows from \cite[Theorem 6.8]{BKNNS}, equivalence $(a)\Leftrightarrow(h)$ and the fact that $A_\S$ is metrizable  (Theorem \ref{TIFS}) and that $\mathfrak{mn}(X)=1$ for metrizable spaces (\cite[p. 1037 and 1038]{BKNNS}).
\end{proof}

\subsection{Fuzzy iterated function systems on topological spaces}\label{sec:fuzzy}
In this section we will present the framework and basic results on the fuzzy approach to the Hutchinson-Barnsley theory. We will extend the metric setting (considered earlier; see \cite{Cabrelli} or \cite{OS}) to a purely topological one.\\
For a topological space $X$, by $\KF(X)$ we denote the family of all fuzzy sets $u:X\to[0,1]$ which are:
\begin{itemize}
    \item \emph{normal}, meaning that $u(x_0)=1$ for some $x_0\in X$;
    \item \emph{upper semicontinuous} (usc), which equivalently means that for every $\alpha\in(0,1]$, the $\alpha$-cut $[u]^\alpha:=\{x\in X:u(x)\geq \alpha\}$ is closed;
    \item \emph{compactly supported}, meaning that their supports 
    $
        [u]^0:=\overline{\{x\in X:u(x)>0\}}
    $
    are compact.
    \end{itemize}
We will call a fuzzy set $u\in\KF(X)$ as \emph{a compact fuzzy set}.\\
In \cite{OS} (see also \cite{Cabrelli}), it was shown that if $(X,d)$ is a metric space, then the space $\KF(X)$ can be metrized by the metric
\begin{equation}\label{dhf}
    d_{HF}(u,v)=\sup_{\alpha\in[0,1]}d_H([u]^\alpha,[v]^\alpha),\;\;\;u,v\in\KF(X),
\end{equation}
which turns out to be complete provided that $d$ is complete (see \cite[Theorem 2.2]{OS}).\\
The fuzzy IFS setting involves the image of a fuzzy set, introduced by Zadeh \cite{Zadeh} and the notion of grey level maps.\\
For a function $f:X\to Y$, where $Y$ is another set, and a fuzzy set $u:X\to[0,1]$, we define the fuzzy set $f[u]:Y\to[0,1]$ in the following way:
\begin{equation}\label{f[u]}
f[u](y):=\left\{\begin{array}{ccc}\sup\{u(x):x\in f^{-1}(y)\}&\mbox{if}&f^{-1}(y)\neq\emptyset\\
0&\mbox{if}&f^{-1}(y)=\emptyset\end{array}\right.
\end{equation}
\begin{remark}\label{rem:2.7}[see \cite[Lemma 3]{KSW} for the metric case]
    \emph{
    If $f:X\to X$ is continuous and $u:X\to [0,1]$ is usc and compactly supported, then
    $$
f[u](y)=\left\{\begin{array}{ccc}\max\{u(x):x\in f^{-1}(y)\}&\mbox{if}&f^{-1}(y)\neq\emptyset\\
0&\mbox{if}&f^{-1}(y)=\emptyset\end{array}\right.
$$
To get this description, one has to use the fact that usc maps on compact sets attain their maxima (see \cite[Theorem 3, p. 361]{Bourbaki}), and restrict the supremum in (\ref{f[u]}) to $f^{-1}(y)\cap[u^0]$.
    }
\end{remark}
    A system $\vr=(\vr_i)_{i=1}^k$ of selfmaps $\vr_i:[0,1]\to[0,1]$ is \emph{an admissible system of grey level maps}, if
    \begin{itemize}
        \item[(i)] for every $i=1,...,k$, the map $\vr_i$ is right continuous and nondecreasing (equivalently, usc and nondecreasing), and satisfies $\vr_i(0)=0$;
        \item[(ii)] there exists $j=1,...,k$, such that $\vr_j(1)=1$.
    \end{itemize}

We are ready to define fuzzy IFSs and their corresponding fuzzy Hutchinson operators.\\ Note that in the theory of fuzzy sets, the maximum of fuzzy sets is a counterpart of the standard union of sets.
\begin{definition}\emph{
    Let $\S=(f_i)_{i=1}^k$ be an IFS on a Hausdorff topological space $X$ and $\vr=(\vr_i)_{i=1}^k$ be an admissible system of grey level maps.
    Then the pair $\S_\F=(\S,\vr)$ will called }a fuzzy iterated function system\emph{ (fuzzy IFS for short)}.\\
    \emph{The operator $\S_\F:\KF(X)\to\KF(X)$ defined by
    $$
        \forall_{u\in\KF(X)}\;\S_\F(u):=\max\{\vr_i\circ(f_i[u]):i=1,...,k\}
    $$
    will be called as} the fuzzy Hutchinson operator \emph{adjusted to $\S_\F$.}
\end{definition}
In \cite{OS} (see also \cite{Cabrelli}) it is proved that the fuzzy Hutchinson operator is well defined if $X$ is a metric space (see \cite[Proposition 2.12]{OS}). The proof of this fact for topological setting is similar. Below we formulate a technical lemma which justifies it and which gives a description of $\alpha$-cuts of a fuzzy Hutchinson operator. Its proof will be short as we will refer to the metric setting if particular reasoning is the same.

\begin{lemma}\label{ophut}
Let $X$ be a Hausdorff topological space.
\begin{itemize}
    \item[(I)] Let $f:X\to X$ be continuous, $u\in\KF(X)$ and $\vr:[0,1]\to[0,1]$ be right continuous, nondecreasing and $\vr(0)=0$.
\begin{itemize}
\item[(i)] Setting $\beta(\alpha):=\left\{\begin{array}{ccc}\inf\{t:\vr(t)\geq \alpha\}&\mbox{if}&0<\alpha\leq\vr(1)\\\inf\{t:\vr(t)>0\}&\mbox{if}&\alpha=0\end{array}\right.
            $, we have
$$[\rho\circ u]^\alpha=\begin{cases}\emptyset & \text{ if } \alpha>\rho(1)\\
                \acut[\beta(\alpha)]{u} & \text{ if } \big(0<\alpha\leqslant\rho(1)\text{ and } \beta(\alpha)>0\big)\;\mbox{or}\;\big(\alpha=0\text{ and }\rho(\beta(0))>0\big)\\
                \overline{\bigcup_{\eta>\beta(0)} [u]^\eta} & \text{ if } \alpha=0 \text{ and } \rho(\beta(0))=0.
                
            \end{cases}
            $$
    \item[(ii)] $\rho\circ u$ is usc and compactly supported.
    \item[(iii)] For every $\alpha\in[0,1]$, \begin{equation}\label{eq:ppoo}[\rho\circ(f[u])]^\alpha=f([\rho\circ u]^\alpha).\end{equation}
\end{itemize}

    \item[(II)] If $\S_\F=((f_i)_{i=1}^k,(\vr_i)_{i=1}^k)$ is a fuzzy IFS and $u\in\K_\F(X)$, then for every $\alpha\in[0,1]$,
   \begin{equation}\label{eq:ppooo}
       [\S_\F(u)]^\alpha=\bigcup_{i=1}^kf_i([\vr_i\circ u]^\alpha).
   \end{equation}
   In particular, $\S_\F(u)\in\KF(X)$.
\end{itemize}
\end{lemma}
\begin{proof}
Item (i) follows the same lines as its metric version - see \cite[Proposition 2.6, items (b) and (c) (here we use the notation $\beta(0)=r_+$)]{OS}.

Item (ii) follows directly from (i) (cf. \cite[Proposition 2.7]{OS}).

The proof of item (iii) follows from similar lines as its metric version (\cite[ Lemma 2.13]{OS}, \cite[Lemma 2.4.1]{Cabrelli}). We give here just a sketch of the proof. If $\alpha\in(0,1]$, then by Remark \ref{rem:2.7}, we have
$$
[\rho\circ(f[u])]^\alpha=\{x\in X:\rho\circ(f[u])(x)\geq\alpha\}=\{x\in X:\rho(\max\{u(y):y\in f^{-1}(x)\})\geq\alpha\}=$$ $$=\{x\in X:\max\{(\rho\circ u)(y):y\in f^{-1}(x)\}\geq\alpha\}=f([\rho\circ u]^\alpha),
$$
so we get (\ref{eq:ppoo}) for $\alpha\in(0,1]$.
If $\alpha=0$, then we proceed similarly. Note that the reasoning handles also the case when both sets are empty.

Finally we deal with (II). Condition (\ref{eq:ppooo}) follows from (iii), as we have
$$
[\S_\F(u)]^\alpha=\{x\in X:\max\{\rho_i\circ(f_i[u])(x):i=1,...,k\}\geq\alpha\}=\bigcup_{i=1}^k\{x\in X:\rho_i\circ(f_i[u])(x)\geq\alpha\}=\bigcup_{i=1}^k[\rho_i\circ(f_i[u])]^\alpha
$$
Finally, item (i) and condition (\ref{eq:ppooo}) justify that $\S_\F(u)\in\KF(X)$ (for normality, recall that $\vr_j(1)=1$ for some $j=1,...,k$).
\end{proof}

The following result can be viewed as a fuzzy version of the Hutchinson theorem.
\begin{theorem}\label{thm:fuzzymat}(\cite[Theorem 3.14 for $m=1$]{OS})\\
    Assume that $\S_\F=(\S,\vr)$ is a fuzzy IFS on a complete metric space $(X,d)$ such that $\S$ is a Matkowski contracting. 
    Then there exists a unique $u_\S\in\KF(X)$ such that
    \begin{itemize}
        \item[(1)] $u_\S=\S_\F(u_\S)$;
        \item[(2)] for any $u\in\KF(X)$, the sequence of iterations $(\S_\F^{(n)}(u))$ converges to $u_\S$ w.r.t. the metric $d_{HF}$.
    \end{itemize}
\end{theorem}
\begin{definition}
    \emph{In the above frame, a (necessarily unique) fuzzy set $u_\S$ satisfying (1) and (2) from the above theorem is called as} a fuzzy attractor\emph{ of $\S_\F$.}
\end{definition}
Recently in \cite{Miculescu} there has been obtained a description of fuzzy attractors which involves the projection map $\pi$. We will formulate it in a short way, for a brief explanation we refer the reader to \cite{Miculescu}. Note that the key result \cite[Theorem 3.4]{Miculescu} is given for Banach contracting fuzzy IFSs, but it also works for Matkowski's ones. Below we use the Zadeh's definition of the image (\ref{f[u]}).
\begin{theorem}\label{thm:2.6}(\cite[Theorems 3.2 and 3.4]{Miculescu})\\
    Let $u_\S$ be a fuzzy attractor of a fuzzy IFS $\S_\F=(\S,\vr)$ on a complete metric space, where $\S$ is Matkowski contracting. Then
    $$
        u_\S=\pi[u_\Lambda],
    $$
    where $\pi:\Sigma_k\to X$ is the projection map and the fuzzy set $u_\Lambda:\Sigma_k\to [0,1]$ is given by
    \begin{equation}\label{eq:canat}
        u_\Lambda(\sigma)=\lim_{n\to\infty}\varrho_{\sigma_1}\circ...\circ\varrho_{\sigma_n}(1),\;\;\;\mbox{for}\;\sigma=(\sigma_n)\in\Sigma_k.
    \end{equation}
\end{theorem}

\section{Fuzzy IFSs and their attractors on multimetric spaces and on Tychonoff spaces}\label{sec:3}

\subsection{The canonical topology on $\KF(X)$ generated by a multimetric}\label{sec:conon}
Let $(X,\D)$ be a multimetric space, and let $\KF(X)$ be the family of compact fuzzy subsets of $X$, when considering the topology $\Tau_\D$ on $X$.
For a pseudometric $d\in\D$, define $d_{HF}$ by
$$\forall_{u,v\in\F(X)}\;d_{HF}(u,v):=\sup\limits_{\alpha\in[0,1]} d_H(\acut{u},\acut{v}),$$
where $d_H$ is defined as in (\ref{eq:fdhh}). 
 Finally, let $$\D_{HF}:=\{d_{HF}:d\in\D\}.$$

\begin{proposition}\label{prop:ppoo}
        In the above frame, the family $\D_{HF}$ is multipseudometric on $\KF(X)$, which is multimetric on $\KF(X)$ provided that $\D$ is directed.
\end{proposition}
\begin{proof}
    Clearly, each $d_{HF}$ is a pseudometric. Assume additionally that $\D$ is directed and take any distinct $u,v\in\KF(X)$. Then there exists $\alpha_0\in[0,1]$ such that $\acut[\alpha_0]{u}\neq\acut[\alpha_0]{v}$. Since $\acut[\alpha_0]{u}, \acut[\alpha_0]{v}\in\K(X)$ and $\D_H$ is a multimetric (see Theorem \ref{directed multimetrics}, item (2)), there exists $d\in\D$ such that $d_H(\acut[\alpha_0]{u},\acut[\alpha_0]{v})>0$. Hence
    $$
        d_{HF}(u,v)=\sup\limits_{\alpha\in[0,1]} d_H\left(\acut[\alpha]{u},\acut[\alpha]{v}\right)\geqslant d_H\left(\acut[\alpha_0]{u},\acut[\alpha_0]{v}\right)>0.
    $$
\end{proof}


\begin{remark}
    \emph{
    In the case when $\D=\{d\}$ for some metric $d$, i.e., when $(X,d)$ is a metric space, then $\D_{HF}=\{d_{HF}\}$, where $d_{HF}$ is defined as in (\ref{dhf}), and the topology $\Tau_{\D_{HF}}$ is generated by the metric $d_{HF}$. This shows that our approach extends the metric case to a multimetric one.
    }
\end{remark}

Similarly as in the case of the hyperspace $\K(X)$, (see Theorem \ref{directed multimetrics}, item (2)), the topology $\Tau_{\D_{HF}}$ does not depend on the choice of a directed multimetric $\D$.
    \begin{theorem}\label{th:rownosc}
        Let $\D$ and $\G$ be directed multimetrics on the same set $X$. If $\Tau_\D=\Tau_\G$, then also $\Tau_{\D_{HF}}=\Tau_{\G_{HF}}$.
    \end{theorem}
    \begin{proof}
        By symmetry of the problem and Lemma \ref{lem:mult1}, it suffices to show that
        $$
        \forall_{v_0\in \KF(X)} \;\forall_{d\in\D} \;\forall_{\varepsilon>0} \;\exists_{g\in\G}\; \exists_{\delta>0}\;  B_{g_{HF}}(v_0, \delta)\subset B_{d_{HF}}(v_0,\varepsilon).
        $$
        Fix $v_0\in\KF(X)$, $d\in\D$ and $\varepsilon>0$. Using Lemma \ref{lem: unversal pseudometric for compact set} for $K=\acut[0]{v_0}$ and $\frac{\varepsilon}{2}$, there exists $g\in\G$ and $\delta>0$, such that for any $x\in \acut[0]{v_0}$, we have
        \begin{equation}\label{eq: thm -> zawieranie kul}
            B_g(x,\delta)\subset B_d\big(x,\tfrac{\varepsilon}{2}\big).
        \end{equation}
        We will show that 
        \[
        B_{g_{HF}}(v_0,\delta)\subset B_{d_{HF}}(v_0, \varepsilon).
        \]
        Let $v\in B_{g_{HF}}(v_0,\delta)$, i.e.,
        \begin{equation}\label{eq: thm -> v jest w kuli}
            \sup\limits_{\alpha\in[0,1]} g_H\left(\acut{v_0},\acut{v}\right)<\delta
        \end{equation}
        and fix any $\alpha\in[0,1]$. For any $x\in \acut{v_0}$, by (\ref{eq: thm -> v jest w kuli}) there exists $y\in \acut{v}$ such that $g(x,y)<\delta$, so $y\in B_g(x,\delta)$. By (\ref{eq: thm -> zawieranie kul}), $y\in B_d(x, \frac{\varepsilon}{2})$. Therefore $\min\limits_{y\in\acut{v}} d(x,y) < \frac{\varepsilon}{2}$ and hence
        \[
        \max\limits_{x\in\acut{v_0}}\min\limits_{y\in\acut{v}} d(x,y) < \frac{\varepsilon}{2}.
        \]
        Now let $x\in \acut{v}$. By (\ref{eq: thm -> v jest w kuli}), there exists $y\in \acut{v_0}$ such that $g(x,y)<\delta$, hence $x\in B_g(y, \delta)$. Now using (\ref{eq: thm -> zawieranie kul}) we have that $x\in B_d(y, \frac{\varepsilon}{2})$. Therefore $\min\limits_{y\in\acut{v_0}} d(x,y) < \frac{\varepsilon}{2}$ and
        \[
        \max\limits_{x\in\acut{v}}\min\limits_{y\in\acut{v_0}} d(x,y) < \frac{\varepsilon}{2}.
        \]
        We conclude that $d_H(\acut{v_0}, \acut{v})<\frac{\varepsilon}{2}$ and, by arbitrariness of $\alpha$, we obtain
        $
        d_{HF}(v_0, v)\leqslant\frac{\varepsilon}{2}<\varepsilon,
    $
        so $v\in B_{d_{HF}}(v_0, \varepsilon)$.
    \end{proof}

\subsection{Fuzzy attractors of contractive fuzzy IFSs on multimetric spaces}\label{sec:2.3}


Similarly as in the case of IFSs, we can consider fuzzy IFSs on multimetric spaces $(X,\D)$ as pairs $\S_\F=(\S,\vr)$, where $\S$ is an IFS on $(X,\D)$.\\
Theorems \ref{directed multimetrics} and \ref{th:rownosc} allow to consider the canonical topology on $\KF(X)$ for multimetric space $(X,\D)$.
\begin{definition}
    \emph{Let $(X,\D)$ be a multimetric space. By } a canonical topology on $\KF(X)$ induced by $\D$\emph{ we will mean the topology $\Tau_{\D'_{HF}}$, where $\D'$ is any directed multimetric on $X$ with $\Tau_\D=\Tau_{\D'}$; we will denote this topology by $\Tau^{\F}_\D$.}
\end{definition}
\begin{remark}
    \emph{Let $(X,\D)$ be a multimetric space.
    \begin{itemize}
        \item[$\bullet$] By Theorem \ref{th:rownosc}, the canonical topology $\Tau_\D^\F$ is well defined, as it does not depend on the choice of directed multimetric $\D'$.
        \item[$\bullet$] By Remark \ref{rem:multimetric} and Proposition \ref{prop:ppoo}, the topology $\Tau_\D^\F$ is Tychonoff (in particular, Hausdorff).
        \item[$\bullet$] The topology $\Tau_\D^\F$ is stronger than the original topology $\Tau_{\D_{HF}}$ (i.e., $\Tau_{\D_{HF}}\subset \Tau_\D^\F$) generated by the multipseudometric $\D_{HF}$, since we can take, as a directed multimetric $\D'$, the directed multimetric $\tilde{\D}$ from Theorem \ref{directed multimetrics}.
        \item[$\bullet$] If $(X,d)$ is a metric space, then  $\Tau_{\{d\}}^\F$ is exactly the topology induced by the metric $d_{HF}$ defined by (\ref{dhf}).
    \end{itemize}
    }
\end{remark}

\begin{definition}\emph{
Let $(X,\D)$ be a multimetric space and $\S_\F$ be a fuzzy IFS.
    A (necessarily unique) compact fuzzy set $u_\S$ is a} fuzzy attractor \emph{of a fuzzy IFS $\S_\F$, if
    \begin{itemize}
        \item[(i)] $u_\S=\S_\F(u_\S)$;
        \item[(ii)] for every $u\in\KF(X)$, the sequence of iterations $(\S_\F^{(n)}(u))$ of the fuzzy Hutchinson operator converges to $u_\S$ w.r.t. the canonical topology $\Tau_{\D}^\F$.
    \end{itemize}
    } 
\end{definition}

Our main result of this section is the following version of the Hutchinson-Barnsley theorem, which is a direct extension of Theorem \ref{thm:fuzzymat}. Recall that the Matkowski contractivity implies the Edelstein contractivity.
\begin{theorem}\label{thm:main}
Let $(X,\D)$ be a multimetric space, and let $\S_\F=(\S,\vr)$ be a fuzzy IFS such that the IFS $\S$ is compactly dominating and Edelstein contracting. Then $\S_\F$ generates a unique fuzzy attractor.
\end{theorem}
We precede the proof with three lemmas. The first one can be found in \cite{BKNNS} (cf. \cite[Proposition 5.4 and Corollary 5.5]{BKNNS} (in fact, we formulate here a particular version that will used later), and we skip its simple proof. The second one was formulated in \cite{OS} only for metric spaces, and we give its short proof. The last one focuses on natural extensions of fuzzy sets on subspaces and we skip its simple proof.
\begin{lemma}\label{lem:5.15}
Let $X$ be a continuous pseudometric on a topological space $X$ and $f:X\to X$. Assume that $K,D\in\K(X)$ and there is a comparison function $\varphi$ such that for every $x,y\in K\cup D$, we have
$
d(f(x),f(y))\leq\varphi(d(x,y))
$. Then
$$
    d_H(f(K),f(D))\leq\varphi(d_H(K,D)).
$$
\end{lemma}
\begin{lemma}\label{lem:3.17}
Let $X$ be a continuous pseudometric on a topological space $X$, and let $\rho:[0,1]\to[0,1]$ be non decreasing and right continuous with and $\rho(0)=0$. Then for every $u,v\in\K_\F(X)$ and $\alpha\in[0,1]$, we have
$$
    d_H([\rho\circ u]^\alpha,[\rho\circ v]^\alpha)\leq d_{HF}(u,v),
$$
where we additionally assume that $d_H(\emptyset,\emptyset)=0$.
\end{lemma}
\begin{proof}
 Similarly as in the metric case, we can prove that for every family $A_i,B_i,\;i\in I$ of compact subsets of $X$, with $\overline{\bigcup_{i\in I}A_i},\overline{\bigcup_{i\in I}B_i}$ compact, we have
    \begin{equation}\label{qqaaq}
    d_{H}\big(\overline{\bigcup_{i\in I}A_i},\overline{\bigcup_{i\in I}B_i}\big)\leq\sup_{i\in I}d_H(A_i,B_i).
    \end{equation}
    Hence the assertion follows from Lemma \ref{ophut}, item (i).
\end{proof}
Now if $X$ is a set, $Y\subset X$ its nonempty subset, $u:Y\to[0,1]$ a fuzzy set on $Y$, then let $e(u):X\to [0,1]$ be defined by
\begin{equation}\label{eq:qqq}
    e(u)(x):=\left\{\begin{array}{ccc}u(x)&\mbox{if}&x\in Y\\0&\mbox{if}&x\notin Y\end{array}\right.
\end{equation}
\begin{lemma}\label{lem:ext}
In the above frame, we have
\begin{itemize}
\item[(i)] If $f:X\to X$ is such that $f{\restriction}_{Y}:Y\to Y$ and $u:Y\to [0,1]$, then $f[e(u)]=e(f{\restriction}_{ Y}[u])$.
\item[(ii)] If $\rho:[0,1]\to[0,1]$ is such that $\rho(0)=0$ and $u:Y\to[0,1]$, then $\rho\circ e(u)=e(\rho\circ u)$.
\item[(iii)] If $u_1,...,u_k:Y\to[0,1]$, then $\max\{e(u_1),...,e(u_k)\}=e(\max\{u_1,...,u_k\})$.
\end{itemize}
\end{lemma}
\begin{proof}(of Theorem \ref{thm:main}).\\
By Theorem \ref{directed multimetrics} and Lemma \ref{rem:2.6}, $\S$ is Edelstein contracting w.r.t. the admissible directed multimetric $\tilde{\D}$. Thus  we can assume that $\D$ is directed and then the topology $\Tau_\D^\F=\Tau_{\D_{HF}}$.\\
Observe first that for every nonempty compact set $K\subset X$ such that $\S(K)\subset K$ and every $u\in\KF(K)$, the natural extension $e(u)\in\KF(X)$, and
\begin{equation}\label{eq:fixext}
    \S_\F(e(u))=e(\tilde{\S}_\F(u)),
\end{equation}
where $\tilde{\S}_\F=(\tilde{\S},\varrho)$ and $\tilde{\S}=(f_i{\restriction}_{ K})_{i=1}^k$ is the restricted IFS. Indeed,
by Lemma \ref{lem:ext}, we have
$$
\S_\F(e(u))=\max\{\rho_i\circ (f_i[e(u)]):i=1,...,k\}=\max\{\rho_i\circ e(f_i{\restriction}_{K}[u]):i=1,...,k\}=
$$
$$
=\max\{e(\rho_i\circ (f_i{\restriction}_{{K}}[u])):i=1,...,k\}=e(\max\{\rho_i\circ (f_i{\restriction}_{{K}}[u]):i=1,...,k\})=e(\tilde{S}_\F(u)).
$$
We are ready to prove the result.
By Theorem \ref{aaa} item (1), the IFS $\S$ is topologically contracting. Hence, by Theorem \ref{aaa}, item (4), we infer that its attractor $A_\S$ is a compact metrizable subset of $X$, and there is an admissible metric $d$ on $A_\S$ making all maps $f_i$ Matkowski contractions.\\
In particular, by Theorem \ref{thm:fuzzymat}, the restricted fuzzy IFS $\hat{\S}_\F=(\hat{\S},\vr)$, where $\hat{S}=(f_i{\restriction}_{ A_\S})_{i=1}^k$, generates a unique fuzzy attractor $u_{\hat{\S}}$.\\
Let $u_*$ be the natural extension of $u_{\hat{\S}}$ to the whole $X$, that is,
$$
    u_*(x)=e(u_{\hat{\S}})(x)=\left\{\begin{array}{ccc}
    u_{\hat{\S}}(x)&\textit{if}&x\in A_\S\\
    0&\textit{if}&x\notin A_\S\end{array}\right..
$$
By (\ref{eq:fixext}), we have
$$
\S_\F(u_*)=\S_\F(e(u_{\hat{S}}))=e(\hat{S}_\F(u_{\hat{S}}))=e(u_{\hat{\S}})=u_*.
$$
Hence it remains to show that for any $u\in\K_\F(X)$, the sequence of iterations $(\S_\F^{(n)}(u))$ converges to $u_*$ w.r.t. the topology $\Tau_{\D_{HF}}$.\\
Take any $u\in \KF(X)$. As $\S$ is topologically contracting, we can find a compact set $K$ such  that $[u]^0\subset K$ and $\S(K)\subset K$. Then automatically $A_\S\subset K$ by uniqueness of an attractor of a topologically contracting IFS.
 Let $\tilde{\S}_\F=(\tilde{\S},\vr)$ for the restricted IFS $\tilde{\S}=(f_i{\restriction}_K)_{i=1}^k$. By (\ref{eq:fixext}) and the fact that $e(u{\restriction}_K)=u$, we have $
e(\tilde{\S}_\F(u{\restriction}_K))=\S_\F(u).$ Then, by an inductive argument, we see that for any $n\in\N$, \begin{equation}\label{eq:ooo}
e(\tilde{\S}_\F^{(n)}(u{\restriction}_K))=\S_\F^{(n)}(u).\end{equation}
In particular, $[\S_\F^{(n)}(u)]^0\subset K$ for every $n\in\N$.

Now fix any $d\in\D$.
By Theorem \ref{aaa} item (2) and the last part or Remark \ref{rem:multimetric}, we have that the restricted IFS $\tilde{\S}=(f_i{\restriction}_{ K})_{i=1}^k$ is Matkowski contracting w.r.t. $d_{K\times K}$ (alternatively, we can use  \cite[Proposition 5.3]{BKNNS}, item (e)). More precisely, there is a comparison function $\varphi_d$ such that for every $i=1,...,k$ and $x,y\in K$, we have
$$
d(f_i(x),f_i(y))\leq \varphi_d(d(x,y)).
$$
Take any $v\in\K_\F(X)$ with $[v]^0\subset K$. Then also $[\S_\F^{(n)}(v)]^0\subset K$ for every $n\in\N$. By Lemmas \ref{ophut}, \ref{lem:5.15} and \ref{lem:3.17}, using condition (\ref{qqaaq}) and again the convention that $d_H(\emptyset,\emptyset)=0$, we have
$$
d_{HF}(\S_\F(u),\S_\F(v))=
\sup_{\alpha\in[0,1]}d_H\big(\bigcup_{i=1}^kf_{i}([\vr_i\circ u]^\alpha),
\bigcup_{i=1}^kf_{i}([\vr_i\circ v]^\alpha\big)\leq
$$
$$
\leq\sup_{\alpha\in[0,1]}\max_{i=1,...,k}
d_H\big(f_{i}([\vr_i\circ u]^\alpha),
f_{i}([\vr_i\circ v]^\alpha)\big)\leq
\sup_{\alpha\in[0,1]}\max_{i=1,...,k}
\varphi_d\big(d_H([\vr_i\circ u]^\alpha,[\vr_i\circ v]^\alpha)\big)\leq
$$
$$
\leq\sup_{\alpha\in[0,1]}\max_{i=1,...,k}\varphi_d\big(d_{HF}(u,v)\big)=\varphi_d\big(d_{HF}(u,v)\big).
$$
Using inductive argument,  inclusions $[\S_\F^{(n)}(u)]^0,[\S_\F^{(n)}(v)]^0\subset K$, and setting $v=u_*$, we have
$$
d_{HF}\big(\S_\F^{(n)}(u),u_*)=d_{HF}\big(\S_\F^{(n)}(u),\S_\F^{(n)}(u_*))\leq \varphi_d^{(n)}\big(d_{HF}(u,u_*)\big)\to 0.
$$
As $d\in\D$ was taken arbitrarily, this implies that $\S_\F^{(n)}(u)\to u_*$ in $\Tau_{\D_{HF}}=\Tau_\D^\F$.
\end{proof}

\subsection{Fuzzy attractors of IFSs on Tychonoff spaces}\label{sec:Tych}
In this section we will reformulate the multimetric setting to the framework of Tychonoff spaces. 

Assume that $X$ is a Tychonoff space. By Remark \ref{rem:multimetric}, its topology $\tau$ is generated by some multimetric $\D$, so we can consider the canonical topology $\Tau_\D^\F$ on $\KF(X)$. This topology does not depend on the choice of admissible multimetric $\D$. Indeed, if $\G$ is another multimetric with $\Tau_\D=\Tau_\G$ and we find directed multimetrics ${\D'}$, ${\G'}$ with $\Tau_\D=\Tau_{{\D'}}$ and $\Tau_\G=\Tau_{{\G'}}$, then $\Tau_{\D'}=\Tau_{\G'}$, so $\Tau_{\D}^\F=\Tau_{\D'_{HF}}=\Tau_{\G'_{HF}}=\Tau_{\G}^\F$. Denote this topology on $\KF(X)$ by $\Tau_M$ (to underline that it comes from multimetrics).

\begin{definition}\emph{
    If $X$ is a Tychonoff space, then the topology $\Tau_M$ on $\KF(X)$, the canonical topology generated by any admissible directed multimetric $\D$ on $X$, will be called as the }M-topology on $\KF(X)$.
\end{definition}

\begin{remark}\emph{
In fact, the $M$-topology is generated by the natural uniformity on $\KF(X)$ which is induced by the universal uniformity on $X$, i.e., the maximal uniformity generating the topology of $X$. This universal uniformity on $X$ is generated by the family of all continuous pseudometrics on $X$.
 Thus we can also call the M-topology also as} strong uniform topology\emph{ on $\KF(X)$. We plan to study this approach to topologization of $\KF(X)$ in a forthcoming paper.}
\end{remark}

\begin{definition}\emph{
Let $X$ be a Tychonoff space and $\S_\F$ be a fuzzy IFS on $X$.
    A (necessarily unique) compact fuzzy set $u_\S$ is a } fuzzy attractor \emph{of a fuzzy IFS $\S_\F$, if
    \begin{itemize}
        \item[(i)] $u_\S=\S_\F(u_\S)$;
        \item[(ii)] for every $u\in\KF(X)$, the sequence of iterations $(\S_\F^{(n)}(u))$ of the fuzzy Hutchinson operator converges to $u_\S$ w.r.t. the M-topology $\Tau_{M}$.
    \end{itemize}
    } 
\end{definition}
\begin{remark}
    \emph{
Directly by definition, we see that a fuzzy set $u_\S$ is a fuzzy attractor of a fuzzy IFS $\S_\F$ on a Tychonoff space $X$ if and only if it is a fuzzy attractor w.r.t. some (equivalently, any) admissible multimetric $\D$ on $X$.}
\end{remark}
Now, Theorem \ref{thm:main} can be reformulated in the following way:
\begin{theorem}\label{thm:main3}
    Let $X$ be a Tychonoff space and $\S_\F=(\S,\vr)$ be a fuzzy IFS on $X$. If $\S$ is Edelstein contracting and compactly dominating w.r.t. some admissible multimetric $\D$ on $X$, then $\S$ admits a fuzzy attractor.
\end{theorem}

It turns out that for so-called $k$-spaces (see \cite[Section 3.3]{En} or \cite[p. 1035]{BKNNS}), we can strengthen the above theorem. Recall that a Hausdorff topological space $X$ is called a \emph{$k$-space}, if its topology is determined by compact sets, in the sense that a subset $U\subset X$ is open if and only if $K\cap U$ is open in $K$ for every compact subset $K\subset X$ (see \cite[3.3.19]{En}). All locally compact spaces and all first countable spaces are $k$-spaces (see \cite[p. 152 and 3.3.20]{En}).
\begin{theorem}\label{thm:main4}
    Let $X$ be a Tychonoff $k$-space and $\S_\F=(\S,\vr)$ be a fuzzy IFS on $X$. If $\S$ is topologically contracting, then $\S_\F$ admits a fuzzy attractor.
\end{theorem}
    
\begin{proof}
By \cite[Theorem 6.3]{BKNNS}, equivalence $(b)\Leftrightarrow(d)$, the IFS $\S$ is Edelstein contracting w.r.t. some admissible multimetric $\D$ on $X$. Hence the result follows from Theorem \ref{thm:main3}.
\end{proof}

\section{Fuzzy IFSs and their attractors on Hausdorff spaces}\label{sec:other}

In the previous discussion, we restricted our framework to Tychonoff spaces $X$ and to M-topologies on $\KF(X)$ generated by directed admissible multimetrics on $\D$. 
In this section, we will extend the setting to fuzzy IFSs on Hausdorff topological spaces. 

\subsection{Topology induced by hypographs}\label{sec:hypo}
Let $X$ be a Hausdorff topological space.
We can define the topology on $\KF(X)$ by identifying fuzzy sets $u\in\K(X)$ with their hypographs
    \begin{equation}\label{eq:hypo}
         \on{hypo}(u):=\{(x,\alpha):x\in[u]^0,\;0\leq\alpha\leq u(x)\}
    \end{equation}
as compact subsets of $X\times[0,1]$.  Thus, we can consider the topology $\Tau_{h}$ on $\KF(X)$ that consists of sets of the form
$$
    \{u\in\KF(X):\on{hypo}(u)\in U\},
$$
where $U$ runs through all open subsets of $\K(X\times [0,1])$; we will call this topology as \emph{$h$-topology}.

Similar approach is discussed for example in \cite{Beer} or \cite{Yang} for a bit different way of defining a hypograph, namely, by
$$
\on{hypo}_0(u):=\{(x,\alpha):x\in X,\;0\leq\alpha\leq u(x)\}
$$
(we will discuss this approach later, in Section \ref{sec:5}). Let us just remark that the $h$-topology $\Tau_h$ is well defined as for any $u\in\KF(X)$, the set $\on{hypo}(u)$ is a compact as intersection of a closed set $\on{hypo}_0(u)$ (see \cite[Proposition 8.3]{Choc} and switch from lsc to usc maps) with a compact set $[u]^0\times [0,1]$. Moreover, $\Tau_h$ is Hausdorff since the Vietoris topology on $\K(X\times[0,1])$ is Hausdorff. Finally, observe that the topology $\Tau_h$ is metrizable if $X$ is metrizable, by a metric $d_h(u,v):=d_H(\on{hypo}(u),\on{hypo}(v))$, where $d$ is any admissible metric on $X\times [0,1]$.

Below we show that for Tychonoff spaces, the  M-topology $\Tau_M$ is always stronger than the $h$-topology $\Tau_h$.

\begin{theorem}\label{thm:zawieranie2}
    If $X$ is a Tychonoff topological space, then $\Tau_h\subset \Tau_M$.
\end{theorem}
    \begin{proof}
    Fix an admissible directed multimetric $\D$ on $X$. Then $\Tau_M=\Tau_{\D_{HF}}$.
        Fix any fuzzy set $u_0\in\KF(X)$ and its subbasic neighbourhood  of one of the forms $$\langle U\rangle^+=\{u\in\KF(X):\on{hypo}(u)\subset U \}\;\;\mbox{or}\;\;\langle U\rangle ^-=\{u\in\KF(X):\on{hypo}(u)\cap U\neq \emptyset\}$$ for some non-empty open set $U$ in $X\times [0,1]$. 
        We will show that in each case, there is a pseudometric $d\in\D$ and $\varepsilon>0$ so that $B_{d_{HF}}(u_0,\varepsilon)$ is contained in this neighborhood.\\$\;$\\
       First we deal with the case $\langle U\rangle^-$.\\
       Since $u_0\in \langle U\rangle^-$, there exists $(y_0,\alpha_0)\in U$ so that $y_0\in[u]^0$ and $0\leq \alpha_0\leq u_0(\alpha_0)$, and thus $y_0\in [u_0]^{\alpha_0}$. As $U$ is open, by the fact that $\D$ is directed, there exists $\varepsilon>0$, $d\in\D$ and $a<b$, such that 
       $$
       (y_0,\alpha_0)\in B_d(y_0,\varepsilon)\times((a,b)\cap[0,1])\subset U.
       $$
       We will show that
       $$
       B_{d_{HF}}(u_0,\varepsilon)\subset \langle U\rangle^-.
       $$
       Take any $u\in B_{d_{HF}}(u_0,\varepsilon)$. As $d_{HF}(u,u_0)<\varepsilon$, we have $d_H([u]^{\alpha_0},[u_0]^{\alpha_0})<\varepsilon$. As $y_0\in [u_0]^{\alpha_0}$, we find $y\in [u]^{\alpha_0}$ with $d(y,y_0)<\varepsilon$. Since $\alpha_0\leq u(y)$ and $y\in[u]^0$, it follows $(y,\alpha_0)\in\on{hypo}(u)\cap U$ and hence $u\in \langle U\rangle^-$.\\$\;$\\
       Now consider the case $\langle U\rangle^+$.\\
       Since $u_0\in \langle U\rangle^+$ and $U$ is open, for every $(x,\alpha)\in\on{hypo}(u_0)$, there exists $d_{x,\alpha}\in\D$, $\varepsilon_{x,\alpha}>0$ and $a_{x,\alpha}<b_{x,\alpha}$ so that
       \begin{equation}\label{aa11}
       (x,\alpha)\in B_{d_{x,\alpha}}(x,\varepsilon_{x,\alpha})\times((a_{x,\alpha},b_{x,\alpha})\cap[0,1])\subset U.
       \end{equation}
       Set $I_{x,\alpha}:=(a_{x,\alpha},b_{x,\alpha})\cap[0,1]$ for $(x,\alpha)\in\on{hypo}(u_0)$. The family
       $$
       \big\{B_{d_{x,\alpha}}\big(x,\frac{\varepsilon_{x,\alpha}}{2}\big)\times I_{x,\alpha}:(x,\alpha)\in\on{hypo}(u_0)\big\}
       $$
       is an open cover of a compact set $\on{hypo}(u_0)$. Hence, we can find its finite subcover 
       $$B_{d_{x_i,\alpha_i}}\big(x_i,\frac{\varepsilon_{x_i,\alpha_i}}{2}\big)\times I_{x_i,\alpha_i},\;\;i=1,...,k.$$
       Since $\D$ is directed, we can find $d\in\D$ with $d\geq\max\{d_{x_i,\alpha_i}:i=1,...,k\}$. Then set $\varepsilon:=\frac{1}{2}\min\{\varepsilon_{x_i,\alpha_i}:i=1,...,k\}$. We will show that
       $$
       B_{d_{HF}}(u_0,\varepsilon)\subset \langle U\rangle^+.
       $$
       Take any $u\in B_{d_{HF}}(u_0,\varepsilon)$ and choose $(y,\alpha)\in\on{hypo}(u)$. Then $y\in[u]^\alpha$. As $d_{HF}(u,u_0)<\varepsilon$, we have $d_H([u]^\alpha,[u_0]^\alpha)<\varepsilon$. In particular, there exists $x\in[u_0]^\alpha$ with $d(x,y)<\varepsilon$.\\
       Since $(x,\alpha)\in \on{hypo}(u_0)$, we find $i=1,...,k$ so that 
       $$
       (x,\alpha)\in B_{d_{x_i,\alpha_i}}\big(x_i,\frac{\varepsilon_{x_i,\alpha_i}}{2}\big)\times I_{x_i,\alpha_i}.
       $$
       Then $\alpha\in I_{x_i,\alpha_i}$ and
       $$
       d_{x_i,\alpha_i}(y,x_i)\leq d_{x_i,\alpha_i}(y,x)+d_{x_i,\alpha_i}(x,x_i)\leq d(y,x)+d_{x_i,\alpha_i}(x,x_i)
       <\varepsilon+\frac{\varepsilon_{x_i,\alpha_i}}{2}\leq \frac{\varepsilon_{x_i,\alpha_i}}{2}+\frac{\varepsilon_{x_i,\alpha_i}}{2}=\varepsilon_{x_i,\alpha_i}.
       $$
       Hence, $y\in B_{d_{x_i,\alpha_i}}\big(x_i,\varepsilon_{x_i,\alpha_i}\big)$ and thus by (\ref{aa11})
       $$
       (y,\alpha)\in B_{d_{x_i,\alpha_i}}\big(x_i,\varepsilon_{x_i,\alpha_i}\big)\times I_{x_i,\alpha_i}\subset U.
       $$
       All in all, $\on{hypo}(u)\subset U$, meaning that $u\in \langle U\rangle^+$.
    \end{proof}

        The inclusion in Theorem \ref{thm:zawieranie2} cannot be strengthened to equality as the following example shows.
        
        \begin{exmp}\label{ex:Tau u neq Tau DHF 1}
        Let $X=[0,1]$ be endowed with the Euclidean metric $d$. Then the $d_{HF}$ induces the $M$-topology on $\KF(X)$. Now consider the following fuzzy sets
        \[
        u_n(x):=\begin{cases}
            1-x, & x\in[0,\frac{1}{n}]\\
            1-\frac{1}{n}, & x\in (\frac{1}{n},1]
        \end{cases}\;\;\;\mbox{for }n\in\N,\;\;\mbox{and}\;\;u_0(x):=1.
        \]
        Then $u_n\to u_0$ in the $h$-topology, as obviously $\on{hypo}(u_n)\to\on{hypo}(u_0)$ in the Vietoris topology on the square $[0,1]^2$ (equivalently, $d_h(\on{hypo}(u_n),\on{hypo}(u_0))\to 0$). On the other hand, $u_n{\nrightarrow} u_0$ in the M-topology as
        \[
        \acut[1]{u_n}=\{0\}\text{ and } \acut[1]{u_0}=[0,1],
        \]
        and hence $d_{HF}(u_n,u_0)\geqslant 1 \nrightarrow 0$.
    \end{exmp}

\subsection{Fuzzy attractors of fuzzy IFSs on Hausdorff topological spaces}\label{sec:TFIFS}
Let $X$ be a Tychonoff space and $\S_\F=(\S,\vr)$ be a fuzzy IFS on $X$. If $u_\S$ is its fuzzy attractor, then by Theorem \ref{thm:zawieranie2}, for any $u\in\KF(X)$, the sequence $\S_\F^{(n)}(u)$ converges to $u_\S$ w.r.t. the $h$-topology $\Tau_h$. Hence and by Example \ref{ex:Tau u neq Tau DHF 1}, the M-topology $\Tau_M$ is better than the $h$-topology $\Tau_h$ from the perspective of properties of fuzzy invariant sets. However, we can use topology $\Tau_h$ to consider fuzzy attractors of IFSs on Hausdorff spaces.

\begin{definition}\emph{
Let $X$ be a Hausdorff topological space and $\S_\F$ be a fuzzy IFS on $X$.
    A (necessarily unique) compact fuzzy set $u_\S$ is a }weak fuzzy attractor \emph{of a fuzzy IFS $\S_\F$, if
    \begin{itemize}
        \item[(i)] $u_\S=\S_\F(u_\S)$;
        \item[(ii)] for every $u\in\KF(X)$, the sequence of iterations $(\S_\F^{(n)}(u))$ converges to $u_\S$ w.r.t. the $h$-topology $\Tau_{h}$.
    \end{itemize}
    } 
\end{definition}
Clearly, if $X$ is Tychonoff space, then each fuzzy attractor is also a weak fuzzy attractor, as $\Tau_h\subset\Tau_M$.
\begin{theorem}\label{thm:main10}
    Assume that $\S_\F=(\S,\vr)$ is a fuzzy IFS on a Hausdorff topological space $X$ so that $\S$ is topologically contracting. Then $\S_\F$ admits a weak fuzzy attractor.
\end{theorem}

\begin{proof}
Let $\hat{\S}_\F=(\hat{\S},\vr)$ be the restricted fuzzy IFS, where $\hat{\S}=(f_i{\restriction}_{A_\S})_{i=1}^k$ and $A_\S$ is an attractor of $\S$, and let $u_{\hat{\S}}$ be a fuzzy attractor of $\hat{\S}_\F$. In the same way as in the proof of Theorem \ref{thm:main} we can show that 
$$
\S_\F(u_*)=u_*,
$$
where $u_*$ is the canonical extension $u_*=e(u_{\hat{\S}})$ (see condition (\ref{eq:qqq}). Indeed, it is easy to see that condition (\ref{eq:fixext}) and further reasoning hold true in a current setting.

Take any $u\in\KF(X)$. By compact dominancy, there exists $K\in\K(X)$ such that $[u]^0\subset K$ and $\S(K)\subset K$. Again, we automatically have that $A_\S\subset K$. By Theorem \ref{thm:main4} and a fact that compact spaces are $k$-spaces, the restricted fuzzy IFS $\tilde{\S}_{\F}=(\tilde{\S},\vr)$, where $\tilde{\S}=(f_i{\restriction}_{K})_{i=1}^k$, admits a fuzzy attractor which necessarily is the canonical extension of $u_{\hat{\S}}$ to the set $K$, or, equivalently, is the restriction $u_*{\restriction}_K$. Indeed, we can follow the same lines as in the beginning of the proof, but starting with the fuzzy IFS $\tilde{\S}_\F$ instead of $\S_\F$, and get that the canonical extension of $u_{\hat{\S}}$ to the set $K$ is $\tilde{\S}_\F$ invariant (it is also worth to note that $A_\S$ is also the attractor of $\tilde{\S}$, so all constructions are the same). 

In particular, the sequence of iterations $\tilde{\S}_{\F}^{(n)}(u{\restriction}_{ K})$ converges to $u_*{\restriction}_K$ w.r.t. M-topology on $\KF(K)$ and, by Theorem \ref{thm:zawieranie2}, w.r.t. the $h$-topology $\Tau^K_h$ on $\KF(K)$ (we added upper index $K$ to clarify that we deal with the space $\KF(K)$).\\
Take any $W\in\Tau_h$ on $\KF(X)$ that contains $u_*$. Then for some open (in the Vietoris topology) set $W'\in\K(X\times [0,1])$, we have
$$
W=\{v\in\KF(X):\on{hypo}(v)\in W'\}
$$
and also $\on{hypo}(u_*)\in W'$. Then $V':=W'\cap\K(K\times[0,1])$ is open in $\K(K\times [0,1])$ w.r.t. the Vietoris topology (see Lemma \ref{lem:subViet}) and since $\on{hypo}(u_*)=\on{hypo}(u_*{\restriction}_ K)$  (as $[u_*]^0\subset A_\S\subset K$), we have
$$
u_*{\restriction}_K\in \{v\in\KF(K):\on{hypo}(v)\in V'\}.
$$
Hence there exists $n_0\in\N$ such that for all $n\geq n_0$, we have $$\on{hypo}\big(\tilde{\S}_{\F}^{(n)}(u{\restriction}_{ K})\big)\in V'.$$
Following the same lines as in the proof of Theorem \ref{thm:main} (involving condition (\ref{eq:fixext})) we can show that condition (\ref{eq:ooo}) also holds in this setting, that is, for any $n\in\N$, $$
e(\tilde{\S}_\F^{(n)}(u{\restriction}_K))=\S_\F^{(n)}(u).
$$
Hence, by compactness of $K$, we have $[e(\tilde{\S}_\F^{(n)}(u{\restriction}_K))]^0=[\S_\F^{(n)}(u)]^0$ and in consequence $$\on{hypo}(\tilde{\S}_{\F}^{(n)}(u{\restriction}_K))=\on{hypo}({\S}_{\F}^{(n)}(u)).$$
Hence $$\on{hypo}({\S}_{\F}^{(n)}(u))\in V'\subset W',$$
which means that ${\S}_{\F}^{(n)}(u)\in W$ for $n\geq n_0$. All in all, the sequence $({\S}_{\F}^{(n)}(u))$ converges to $u_*$ w.r.t. the topology $\Tau_h$ and $u_*$ is a weak fuzzy attractor of $\S_\F$.
\end{proof}
\subsection{A description of a (weak) fuzzy attractor}\label{sec:descr}
In this part we give a description of a (weak) fuzzy attractor. We have already observed (in the proofs of Theorems \ref{thm:main} and \ref{thm:main10}) that a (weak) fuzzy attractor is the canonical extension of a fuzzy attractor of a certain restricted fuzzy IFS, and (in the proof of Theorem \ref{thm:main10}), that its restrictions to subinvariant compact sets are fuzzy attractors of restricted fuzzy IFSs. Let us formulate these facts as separate observation.

\begin{proposition}\label{prop:main1}
    Let $\S_\F=(\S,\vr)$ is a fuzzy IFS so that $\S=(f_i)_{i=1}^k$ is topologically contracting  and let $u_\S$ be its weak fuzzy attractor.
    \begin{itemize}
        \item[(i)] A fuzzy attractor $u_\S$ is the canonical extension
    $$
    u_\S=e(u_{\hat{\S}}),
    $$
    where $u_{\hat{\S}}$ is a fuzzy attractor of 
    the fuzzy IFS $\hat{\S}_\F=(\hat{\S},\vr)$, where $\hat{\S}=(f_i{\restriction}_{A_{{\S}}})_{i=1}^k$ and $A_{\S}$ is an attractor of  ${{\S}}$.
    \item[(ii)] If $K\in\K(X)$ is such that $\S(K)\subset K$ and $\tilde{\S}_\F=(\tilde{\S},\vr)$ for $\tilde{\S}=(f_{i}{\restriction}_K)_{i=1}^k$, then the restriction $u_\S{\restriction}_K$ is a fuzzy attractor of $\tilde{\S}_\F$.
    \end{itemize}
    In particular, these facts hold true for a fuzzy attractor of a fuzzy $IFS$ on a multimetric space $(X,\D)$ which is compact dominating and Edelstein contracting. 
\end{proposition}
Finally we show that for a weak fuzzy attractor of a fuzzy IFS $\S_\F=(\S,\vr)$ with topologically contracting IFS $\S$, we have the same description as the one presented in Theorem \ref{thm:2.6}. Note that the existence of a canonical projection $\pi$ for topologically contracting IFSs is justified in Remark \ref{rem:prof2}.
\begin{theorem}
    If $\S_\F=(\S,\vr)$ is a fuzzy IFS so that $\S$ is topologically contracting, then its weak fuzzy attractor $u_\S$ is of the form
    $$
        u_\S=\pi[u_\Lambda],
    $$
    where $u_\Lambda$ is given by (\ref{eq:canat}) and $\pi:\Sigma_k\to X$ is the canonical projection.\\
    In particular, this description holds true for a fuzzy attractor of a fuzzy $IFS$ on a multimetric space $(X,\D)$ which is compact dominating and Edelstein contracting. 
\end{theorem}
\begin{proof}
Let $\hat{\S}_\F=(\hat{\S},\vr)$, $u_{\hat{\S}}$ and $A_{{\S}}$ be as in Proposition \ref{prop:main1}, item (i). By Theorem \ref{thm:2.6}, we have that
$$
    u_{\hat{\S}}=\hat{\pi}[u_\Lambda],
$$
where $\hat{\pi}$ is the canonical projection associated to $\hat{\S}$. Clearly, $\hat{\pi}=\pi$ as maps. Hence and by Proposition \ref{prop:main1}, we have that
$$
    u_\S=e(u_{\hat{\S}})=e(\hat{\pi}[u_\Lambda])=\pi[u_\Lambda].
$$
In the last equality we used the fact that $\pi(\Sigma_k)=A_\S$.
\end{proof}
As a corollary of Proposition \ref{prop:main1} we can get a description of a fuzzy (weak) attractor for trivial admissible systems of grey level maps. Recall that for a set $C\subset X$, by $\chi_C$ we denote the characteristic function of $C$.
\begin{proposition}
    Let $\S_\F=(\S,\vr)$ is a fuzzy IFS so that $\S=(f_i)_{i=1}^k$ is topologically contracting and $\vr=(\vr_i)_{i=1}^k$ is such that $\vr_i(1)=1$ for any $i=1,...,k$. Then a weak fuzzy attractor $u_\S$ equals
    $$
    u_\S=\chi_{A_\S}.
    $$
\end{proposition}
\begin{proof}
By \cite[Corollary 4.6 and Remark 4.7, for $m=1$]{OS}, if $\vr_i(1)=1$ for all $i=1,...,k$, then a fuzzy attractor of the restricted fuzzy IFS $\hat{\S}_\F=(\hat{\S},\vr)$ for $\hat{\S}=(f_i{\restriction}_{A_\S})_{i=1}^k$, equals $\chi_{A_\S}$ on $A_\S$. Thus Proposition \ref{prop:main1}, item (i), implies the assertion.
\end{proof}

\section{Other topologies on the space $\KF(X)$}\label{sec:5}
\subsection{An alternative version of a topology generated by hypographs}\label{sec:5.1}
Let $X$ be a Hausdorff topological space.\\
We can consider another topology on $\KF(X)$ by identifying fuzzy sets $u\in\KF(X)$ with their hypographs defined by $\on{hypo}_0(u)$, mentioned in Subsection \ref{sec:hypo}:
$$
    \on{hypo}_0(u):=\{(x,\alpha):x\in X,\; 0\leq\alpha\leq u(x)\}=\on{hypo}(u)\cup(X\times\{0\})
$$
as closed (not necessarily compact) subsets of $X\times [0,1]$.
More precisely, let the topology $\Tau_{h_0}$, called as the $h_0$-topology, consists of sets of the form $$\{u\in\KF(X):\on{hypo}_0(u)\in U\},$$ for open sets $U\subset \mathcal{C}(X\times [0,1])$, where we endow the space $\mathcal{C}(X\times [0,1])$ of all closed and nonempty subsets of $X\times [0,1]$ with the Vietoris topology (see \cite[2.7.20]{En}). Clearly, if $X$ is additionally compact, then $\mathcal{C}(X\times [0,1])=\K(X\times[0,1])$. 

It turns out that the $h_0$-topology $\Tau_{h_0}$ is weaker than the topology $\Tau_{h}$.
\begin{proposition}
    In the above frame, $\Tau_{h_0}\subset \Tau_{h}$.
\end{proposition}
\begin{proof}
Take a subbasic open set in $\Tau_{h_0}$ of one of the forms
$$\langle U\rangle^+=\{u\in\KF(X):\on{hypo}_0(u)\subset U \}\;\;\;\mbox{or}\;\;\;\langle U\rangle ^-=\{u\in\KF(X):\on{hypo}_0(u)\cap U\neq \emptyset\}$$ for some non-empty open set $U$ in $X\times [0,1]$.\\$\;$\\
First deal with the case $\langle U\rangle^+$ and choose any its element $u_0$.\\
Then $\on{hypo}_0(u_0)\subset U$ and in particular $X\times\{0\}\subset U$. Hence $W:=\{u\in\KF(X):\on{hypo}(u)\subset U\}\in\Tau_h$ and $u_0\in W\subset \langle U\rangle^+$.\\$\;$\\
Now consider $\langle U\rangle^-$ and choose any its element $u_0$.\\
If $\on{hypo}(u_0)\cap U=\emptyset$, then $(x,0)\in U$ for some $x\in X$ and hence $\langle U\rangle^-=\KF(X)\in\Tau_h$. If $\on{hypo}(u_0)\cap U\neq\emptyset$, then $W:=\{u\in\KF(X):\on{hypo}(u)\cap U\neq\emptyset \}\in\Tau_h$ and $u_0\in W\subset \langle U\rangle^-$.
\end{proof}

The example below show that the inclusion is usually strict. Recall that for a point $c\in X$, the Dirac function $\chi_c$ of $c$ is the characteristic function of $\{c\}$.
\begin{exmp}\label{ex:aass} Assume that $X$ is a Hausdorff topological space with at least two elements $a,b$, and consider the function $\chi_a+\frac1n\chi_b$. It is easy to see that the sequence $(\chi_a+\frac1n\chi_b)_{n=1}^\infty$ convergence to the Dirac function $\chi_a$ in the $h_0$-topology, but not in the $h$-topology. Consequently, $\Tau_{h_0}\subsetneq\Tau_h$.
\end{exmp}

The above shows that the $h$-topology $\Tau_h$ is better than the $h_0$-topology $\Tau_{h_0}$ from the perspective of behaviour of invariant fuzzy sets, as convergence w.r.t. $\Tau_h$ implies the convergence w.r.t. $\Tau_{h_0}$. On the other hand, the good property of the $h_0$-topology $\Tau_{h_0}$ is that it preserves compactness, which need not be the case for the topology $\Tau_h$.
\begin{proposition}
    Let $X$ be a Hausdorff topological space. Then $X$ is compact iff  $(\KF(X),\Tau_{h_0})$ is compact.
\end{proposition}
\begin{proof} 
Assume first that $X$ is compact.
By definition of the topology $\Tau_{h_0}$, it is enough to prove that the set
$$
\K_{h_0}:=\{K\in\K(X\times [0,1]):K=\on{hypo}_0(u)\;\mbox{for some }u\in\KF(X)\}
$$
is closed in $\K(X\times[0,1])$ endowed with the Vietoris topology (which is compact as $X\times[0,1]$ is compact).\\
Observe that a compact set $\emptyset\neq K\subseteq X\times[0,1]$ does not belong to $\mathcal K_{h_0}$ iff one of the following conditions hold:
\begin{itemize}
    \item[$(i)$] $(x,0)\notin K$ for some $x\in X$,
    \item[$(ii)$] there exists $x\in X$ and $0\leq t<\alpha$ such that $(x,\alpha)\in K$ and $(x,t)\notin K$, or
    \item[$(iii)$] $(x,1)\notin K$ for all $x\in X$.
\end{itemize}
Take any $K\notin \K_{h_0}$. If $(i)$ holds, then the set $\{C\in \mathcal K(X\times[0,1]):(x,0)\notin C\}$ is an open neighborhood of $K$, disjoint from $\mathcal K_{h_0}$. If $(ii)$ holds, then choose an open neighborhood $U_x$ of $x$ in $X$ such that $(\overline U_x\times\{t\})\cap K=\emptyset$ (we can take such a set $U_x$ as $X$ is regular, being compact Hausdorff) and observe that $$\{C\in \mathcal K(X\times[0,1]):(\overline U_x\times\{t\})\cap C=\emptyset\ne (U_x\times (t,1])\cap C\}$$ is an open neighborhood of $K$, disjoint from $\mathcal K_{h_0}$. Finally, if $(iii)$ holds, then the set
$\{C\in\K(X\times[0,1]):C\subset X\times[0,1)\}$ is an open neighborhood of $K$, disjoint from $\mathcal K_{h_0}$. All in all, $\K_{h_0}$ is closed, hence compact.

Now assume that $(\KF(X),\Tau_{h_0})$ is compact. Identifying the space $X$ with the closed subspace $\{\chi_x:x\in X\}$ of $(\KF(X),\Tau_{h_0})$, we conclude that the space $X$ is compact. 
\end{proof}
\begin{exmp} Example \ref{ex:aass} shows that if $X$ has at least two points, then the topological space $(\KF(X),\Tau_{h})$ is not compact. Indeed, the sequence $(\chi_a+\frac1n\chi_b)_{n=1}^\infty$ has no accumulation points in the space $(\KF(X),\Tau_h)$, witnessing that this space is not (countably) compact (see \cite[3.10.3]{En}).
\end{exmp}

\subsection{The pointwise and the uniform topologies on $\KF(X)$}\label{sec:unif}
Let $X$ be a Hausdorff topological space. Looking at the set $\KF(X)$ as a subset of the family $[0,1]^X$ of all functions from $X$ to $[0,1]$, we can consider topologies on $\KF(X)$ induced by natural topologies from $[0,1]^X$. We will consider two ,,extremal'' ones - the \emph{uniform convergence topology} $\Tau_u$, induced by the $\ell_\infty$-metric 
$$
d_{\infty}(u,v):=\sup\{|u(x)-v(x)|:x\in X\}, \;\;u,v\in\KF(X),
$$
and the \emph{pointwise convergence topology} $\Tau_p$, whose subbasis consists of sets of the form
$$
\{u\in\KF(X):|u(x_0)-u_0(x_0)|<\varepsilon\},
$$
where $u_0\in\KF(X)$, $x_0\in X$ and $\varepsilon>0$. Clearly, $\Tau_p\subset\Tau_h$. 
The following examples show that these topologies are not good for our purposes.
In the first one we observe that the topologies $\Tau_p$, $\Tau_u$ and the M-topology $\Tau_M$ may be incomparabe.
\begin{exmp}\label{Tau u neq Tau DHF 1}
         Let $X$ be any Tychonoff space.

Assume that $X$ has two distinct points $a,b$ in $X$. Then the sequence $(\chi_a+\frac1n\chi_b)_{n=1}^\infty$ converges to the Dirac function $\chi_a$ in the $\ell_\infty$-metric on $\KF(X)$, but does not converge to $\chi_a$ in the topology $\Tau_M$ (or $\Tau_h$). Therefore,
 $\Tau_M\nsubseteq\Tau_u$ and hence  $\Tau_M\nsubseteq\Tau_p$.

Now assume that $X$ has a point $a$ which is not isolated in $X$. Then the set $\{u\in\KF(X):u(a)>\frac12\}$ is a neighborhood of $\chi_a$ in the topology $\Tau_p\subset\Tau_u$, but this set is not a neighborhood of $a$ in the topology $\Tau_M$ (because of the continuity of the function $X\to\KF(X)$, $x\mapsto\chi_x$, in the $M$-topology on $\KF(X)$).

    \end{exmp}
In the next example we observe that in a simple case of Banach contractive fuzzy IFSs on metric spaces, we cannot guarantee the convergence of sequences of iterations of a fuzzy Hutchinson operator to a fuzzy attractor w.r.t. the topology $\Tau_p$ (in particular, w.r.t. the topology $\Tau_u$ or the metric $d_{\infty}$).

\begin{exmp}
Let again $X=[0,1]$ be endowed with the Euclidean metric $d$, and consider the fuzzy IFS $\S_\F=(\{f\},\{\rho\})$, where $f(x):=\frac{1}{2}x$ and $\rho(x):=x$ for $x\in[0,1]$. 
Then for any $u\in\KF(X)$, we have
\begin{equation}\label{kl}
\S_\F(v)(y)=\rho(f[v](y))=f[v](y)=\left\{\begin{array}{ccc}\sup\{v(x):x\in f^{-1}(y)\}&\mbox{if}&y\in[0,\frac{1}{2}]\\0&\mbox{if}&y\in(\frac{1}{2},1]\end{array}\right.=\left\{\begin{array}{ccc}v(2y)&\mbox{if}&y\in\big[0,\frac{1}{2}\big]\\0&\mbox{if}&y\in\big(\frac{1}{2},1\big]\end{array}\right..
\end{equation}
According to Theorem \ref{thm:fuzzymat}, the fuzzy IFS $\S_\F$ generates a fuzzy attractor $u_\S$, and by (\ref{kl}), it is easy to see that it is the characteristic function of $\{0\}$ is its fuzzy attractor, $u_\S=\chi_0$.

Now let $u=\chi_1$.
Using  (\ref{kl}) and an inductive argument, we see that for every $n\in\N$,
$
\S_\F^{(n)}(u)=\chi_{\frac{1}{2^n}}
$.
Hence $(\S_\F^{(n)}(u))$ does not converge to $u_{\S}$ w.r.t. the $\Tau_p$ topology.
\end{exmp}

\bibliographystyle{amsplain}

\end{document}